\documentclass[oneside]{amsart}
\usepackage[T1]{fontenc}
\usepackage[latin9]{inputenc}
\usepackage{mathrsfs}
\usepackage{amsthm}
\usepackage{amstext}
\usepackage{amssymb}
\usepackage{comment}

\makeatletter
\numberwithin{equation}{section}
\numberwithin{figure}{section}
\theoremstyle{plain}
\newtheorem{thm}{\protect\theoremname}
\newtheorem{crl}{Corollary}
\newtheorem{dfn}{Definition}
\newtheorem{prp}{Proposition}
\newtheorem{lemma}{Lemma}
\newtheorem{question}{Question}

\makeatother

\usepackage{babel}
\providecommand{\theoremname}{Theorem}

\author{J. Cancino-Manr\'iquez}
\address[Jonathan Cancino-Manr\'{i}quez]{Departamento de Matem\'aticas, Facultad de Ciencias, Universidad Nacional Aut\'onoma de M\'exico. Circuito Exterior S/N, Ciudad Universi\-ta\-ria, CDMX, 04510, M\'exico}
\email{jcancino@ciencias.unam.mx,mhacajoh@gmail.com}

\thanks{
\emph{Acknowledgments.} This work was partly supported by
 Universidad Nacional Aut\'onoma de M\'exico Postdoctoral
 Program (POSDOC). The author is also thankful to Osvaldo Guzm\'an for listening earlier versions of these results.
}



\begin{document}

\title{$Con(\mathfrak{r}_\mathsf{nwd}<\mathfrak{irr})$}
\begin{abstract}
We prove the consistency of the inequality $\mathfrak{r}_{\mathsf{nwd}}<\mathfrak{irr}$, which in turn implies the consistency of 
$\mathfrak{r}_\mathsf{nwd}<\mathfrak{i}$ and $\mathfrak{r}_{\mathsf{scatt}}<\mathfrak{irr}$. This answers one question from \cite{balzar_hrusak_hernandez} and one question from \cite{cancino_irresolvable_1}. We also prove the consistency of the inequality $\mathfrak{r}_\mathbb{Q}<\mathfrak{u}_\mathbb{Q}$.
\end{abstract}

\maketitle

\section{Introduction}

Recall that a topological space $(X,\tau)$ is resolvable if it has a dense co-dense subset; otherwise, the space $(X,\tau)$ is said to be irresolvable. The rationals are the typical example of a resolvable space, as it is easy to construct a dense codense subset of $\mathbb{Q}$. However, the construction of a countable irresolvable space needs some strength of the Axiom of Choice: in \cite{resolvability_and_extraresolvability}, it has been proved that in the Solovay model there is no countable irresolvable space. The argument goes roughly as follows: the existence of an countable irresolvable space implies the existence of a countable irresolvable space such that any open subset is irresolvable with the subspace topology(such spaces are called strongly irresolvable). On the other hand, the ideal of nowhere dense subset of a countable strongly 
irresolvable space is saturated, and therefore, it does not have the Baire property. On the other hand, in the Solovay model every subset of the reals has the Baire property.

The irresolvability number was introduced by M. Scheepers in \cite{scheepers_games}, and it is defined as 
\begin{equation*}
    \mathfrak{irr}=\min\{\pi w((\omega,\tau)):(\omega,\tau)\text{ is an irresolvable  $T_3$ topological space}\}
\end{equation*}

Where $\pi w$ denotes the $\pi$-weight of the space. It is well known that $\mathfrak{r}\leq\mathfrak{irr}\leq\mathfrak{i}$ (see \cite{scheepers_games}). On the other hand, in \cite{cancino_irresolvable_1} it has been proved that $\mathfrak{d}\leq\mathfrak{r}_{\mathsf{scattered}}\leq\mathfrak{irr}$ follows from $\mathsf{ZFC}$, where $\mathfrak{r}_{\mathsf{scattered}}$ is the reaping number of the boolean algebra $\mathcal{P}(\mathbb{Q})/\mathsf{scattered}$, and $\mathsf{scattered}$ is the ideal of scattered subsets of $\mathbb{Q}$. Also, the parametrized diamond principle $\diamondsuit(\mathfrak{r}_{\mathsf{scattered}})$ implies the existence of a countable irresolvable space of weight $\omega_1$. Among other questions, it was asked if the inequality $\mathfrak{r}_{\mathsf{scattered}}<\mathfrak{irr}$ is relatively consistent with $\mathsf{ZFC}$. Interestingly, there is a parallelism with results from \cite{balzar_hrusak_hernandez}. Define $\mathfrak{r}_\mathbb{Q}$ as the minimal cardinality of a dense reaping family, that is, a family $\mathcal{R}$ of dense subsets of the rationals such that for any dense $A\subseteq\mathbb{Q}$, there is $B\in\mathcal{R}$ such that either $B\setminus A$ is not dense or $A\cap B$ is not dense. In \cite{balzar_hrusak_hernandez}, the authors proved that $\mathsf{cof}(\mathcal{M})\leq\mathfrak{r}_{\mathbb{Q}}\leq\mathfrak{i}$, and $\diamondsuit(\mathfrak{r}_{\mathbb{Q}})$ implies $\mathfrak{i}=\omega_1$. The question of whether the inequality $\mathfrak{r}_{\mathbb{Q}}<\mathfrak{i}$ is consistent remained open. In this article we give answer to both questions. The main theorem of this paper is the following:

\begin{thm}
It is relatively consistent with $\mathsf{ZFC}$ that $\mathfrak{r}_{\mathsf{scattered}}=\mathfrak{r}_{\mathbb{Q}}=\omega_1<\omega_2=\mathfrak{irr}=\mathfrak{i}=2^{\omega}$.
\end{thm}

The previous result can be easily adapted to prove that $\mathfrak{r}_\mathbb{Q}<\mathfrak{u}_\mathbb{Q}$ is consistent:

\begin{thm}
The inequality $\mathfrak{r}_\mathbb{Q}<\mathfrak{u}_{\mathbb{Q}}$ is relatively consistent with $\mathsf{ZFC}$. 
\end{thm}

The strategy in proving Theorem 1 is a mixture of the consistency proofs of the inequalities $\mathfrak{r}<\mathfrak{u}$ and $\mathfrak{i}<\mathfrak{u}$. The forcing we use is that of \cite{shelah_i_less_u} used to prove the consistency of $\mathfrak{i}<\mathfrak{u}$. We perform a countable support iteration of length $\omega_2$, and in each step of the iteration, we destroy a countable irresolvable space by adding a dense co-dense subset of $\omega$. Then we take care that any countable irresolvable space that appears along the iteration is destroyed at some step of the iteration. Since once a dense codense subset of $\omega$ has been added to a topological space it can not be taken away, when we destroy the irresolvability of a topological space it can not be restored as an irresolvable space unless we modify the base of the topology. In this way, if we destroy the irresolvabiblity of a topological space, it will remain as a resolvable space after the full iteration, making sure that all countable topological spaces of small $\pi$-weight are resolvable. On the other hand, we make sure that at any step of the iteration we destroy a small quantity of maximal filters on $Dense(\mathbb{Q})$, thus preserving the dense subsets of the ground model as a $Dense(\mathbb{Q})$-reaping family after the iteration. The paper ends with a section discussing the existence of rational $p$-filters and rational $q$-filters.

\section{Preliminaries and basic notions.}

The family of dense subsets of the rationals is denoted by $Dense(\mathbb{Q})$, while the ideal of nowhere dense subsets of $\mathbb{Q}$ is denoted as $\mathsf{nwd}$. The ideal of scattered subsets of the rationals is denoted by $\mathsf{scattered}$.

\begin{dfn}[See \cite{balzar_hrusak_hernandez}]
A family $\mathcal{R}\subseteq Dense(\mathbb{Q})$ is $Dense(\mathbb{Q})$-reaping if for any set $X\in Dense(\mathbb{Q})$, there is $Y\in\mathcal{R}$ such that either:
\begin{enumerate}
    \item $Y\setminus X\notin Dense(\mathbb{Q})$.
    \item $X\cap Y\notin Dense(\mathbb{Q})$.
\end{enumerate}
The cardinal invariant $\mathfrak{r}_\mathbb{Q}$ is defined as the minimal possible cardinality of a $Dense(\mathbb{Q})$-reaping family.
\end{dfn}

\begin{dfn}
A family $\mathcal{R}\subseteq Dense(\mathbb{Q})$ is $Dense(\mathbb{Q})^+$-reaping if for any $X\in Dense(\mathbb{Q})$, there is $Y\in\mathcal{R}$ such that either:
\begin{enumerate}
    \item $Y\subseteq^*X$.
    \item $X\cap Y\notin Dense(\mathbb{Q})$.
\end{enumerate}
The cardinal invariant $\mathfrak{r}_\mathbb{Q}^+$ is defined as the minimal cardinality of a $Dense(\mathbb{Q})^+$-reaping family.
\end{dfn}

Obviously, $\mathfrak{r}_{\mathbb{Q}}\leq\mathfrak{r}_\mathbb{Q}^+$ holds. Note that there exist $Dense(\mathbb{Q})^+$-reaping families which are not trivial: let $\mathcal{F}\subseteq Dense(\mathbb{Q})$ be a maximal filter, that is, for a dense set $X\subseteq \mathbb{Q}$, $X\in\mathcal{F}$ if and only if for all $Y\in\mathcal{F}$, $X\cap Y\in Dense(\mathbb{Q})$. An application of the Axiom of Choice shows that such maximal filters exist. Clearly, any such filter is a $Dense(\mathbb{Q})^+$-reaping family (or any base for such a filter). In the following sections we will see that consistently these families can have cardinality smaller than the continuum. We don't know if $\mathfrak{r}_\mathbb{Q}=\mathfrak{r}_\mathbb{Q}^+$ holds.

\begin{dfn}
$\mathcal{P}(\mathbb{Q})/\mathsf{nwd}$ denotes the boolean algebra of the power set of $\mathbb{Q}$ modulo the ideal of nowhere dense subsets of $\mathbb{Q}$, with the order given by set inclusion modulo $\mathsf{nwd}$. Similarly, $\mathcal{P}(\mathbb{Q})/\mathsf{scattered}$ denotes the quotient modulo the ideal $\mathsf{scattered}$, with the order given by set inclusion modulo $\mathsf{scattered}$.

We denote by $\mathfrak{r}_{\mathsf{nwd}}$ the reaping number of the quotient algebra $\mathcal{P}(\mathbb{Q})/\mathsf{nwd}$, that is, $$\mathfrak{r}_{\mathsf{nwd}}=\mathfrak{r}(\mathcal{P}(\mathbb{Q})/\mathsf{nwd})$$

The reaping number of the boolean algebra $\mathcal{P}(\mathbb{Q})/\mathsf{scattered}$ is denoted by $\mathfrak{r}_{\mathsf{scattered}}$, $$\mathfrak{r}_{\mathsf{scatterd}}=\mathfrak{r}(\mathcal{P}(\mathbb{Q})/\mathsf{scattered})$$
\end{dfn}

\begin{prp}[See \cite{balzar_hrusak_hernandez}]
$\mathfrak{r}_{\mathbb{Q}}=\mathfrak{r}_{\mathsf{nwd}}$.    
\end{prp}

\begin{dfn}
Given two functions $f,g\in\omega^\omega$, define $f\leq^* g$ if the set $$\{n\in\omega:g(n)< f(n)\}$$ is finite.    

A family $\mathcal{D}\subseteq\omega^\omega$ is $\leq^*$-dominating if for any $f\in\omega^\omega$, there is $g\in\mathcal{D}$ such that $f\leq^* g$.

The dominating number $\mathfrak{d}$ is the minimal possible cardinality of a $\leq^*$-dominating family.
\end{dfn}

\begin{prp}[See \cite{balzar_hrusak_hernandez}]
$\mathfrak{d}\leq\mathsf{cof}(\mathcal{M})\leq\mathfrak{r}_{\mathsf{nwd}}$.
\end{prp}

\begin{dfn}
Let $(X,\tau)$ be a regular topological space. We say that $(X,\tau)$ is irresolvable if the intersection of any two dense sets is non-empty.
\end{dfn}

\begin{dfn}
The irresolvability number $\mathfrak{irr}$ is defined as the minimal $\pi$-weight of a countable regular  irresolvable space,
\begin{equation*}
    \mathfrak{irr}=\min\{{\pi}w(\omega,\tau):(\omega,\tau)\text{ is a regular irresolvable space}\}
\end{equation*}
\end{dfn}

\begin{prp}[See \cite{balzar_hrusak_hernandez} and \cite{cancino_irresolvable_1}]\quad

\begin{enumerate}
    \item $\max\{\mathfrak{d},\mathfrak{r}\}\leq\mathfrak{r}_\mathsf{scatt}\leq\mathfrak{irr}\leq\mathfrak{i}$.   \item $\mathfrak{r}_{\mathsf{scatt}}\leq\mathfrak{r}_\mathsf{nwd}\leq\mathfrak{i}$.
\end{enumerate}

\end{prp}

The following are the questions we are answering.

\begin{question}[Question (3) from \cite{cancino_irresolvable_1}]
Is there a model where $\mathfrak{r}_{\mathsf{scattered}}<\mathfrak{irr}$?
\end{question}

\begin{question}[Question 3.11(5) from \cite{balzar_hrusak_hernandez}]
Is $\mathfrak{r}_{\mathbb{Q}}<\mathfrak{i}$ consistent?
\end{question}

\section{Filters of dense subsets of the rationals.}

\begin{dfn}\label{q_filter_def}
Let $(Dense(\mathbb{Q}),\subseteq^*)$ denote the partial order of dense subsets of the rationals with the order given by set inclusion modulo finite sets ($\subseteq^*$). We define the following:
\begin{enumerate}
    \item A maximal $\mathcal{F}\subseteq Dense(\mathbb{Q})$ will be called a rational filter. 
    \item A rational filter which is a $q$-filter will be called a rational $q$-filter.
    \item A rational filter which is a $p$-filter will be called a rational $p$-filter.
    \item A selective filter is a filter which is a $p$-filter and a $q$-filter at the same time.
    \item A rational selective filter is a rational filter which is selective.
\end{enumerate}
\end{dfn}

Usually we will identify $\mathbb{Q}$ with $\omega$ having a topology homeomorphic to the usual topology on $\mathbb{Q}$. In what follows, and for the rest of the paper, we denote by $\mathcal{B}$ a countable base for such topology, which will remain fixed. We also convey that whenever $\langle I_n:n\in\omega\rangle$ is a partition of $\omega$ into intervals, the enumeration is such that $\max(I_n)<\min(I_{n+1})$ for all $n\in\omega$.
 
We being recalling the following well known facts.

\begin{lemma}
Let $\mathcal{F}$ be a rational filter. Then:
\begin{enumerate}
    \item $\mathcal{F}$ is non-meager.
    \item If $\mathcal{F}$ is a $p$-filter, then for any decreasing $\langle B_n:n\in\omega\rangle\subseteq\mathcal{F}$, there is $X\in\mathcal{F}$ such that for infinitely many $n\in\omega$, $X\setminus n\subseteq B_n$.
    \item If $\mathcal{F}$ is a $q$-filter, then for any interval partition $\mathcal{I}=\langle I_n:n\in\omega\rangle$ of $\omega$ such that for all $n\in\omega$, $\max(I_n)<\min(I_{n+1})$, and for any $l\in\omega$, there is a partial selector $X\in\mathcal{F}$ of $\mathcal{I}$ such that for any $n,m\in\{k\in\omega:X\cap I_k\neq\emptyset\}$, $\vert n-m\vert\ge l$.
\end{enumerate}
\end{lemma}

\begin{proof}
\begin{enumerate}
    \item [(1)] We prove that $\mathcal{F}$ is the intersection of $\omega$ many ultrafilters. Indeed, for any $B\in\mathcal{B}$,
    let $\mathcal{U}_B$ be an ultrafilter on $\mathbb{Q}$ extending $\mathcal{F}\cup\{B\}$. Then $\mathcal{F}=\bigcap_{B\in\mathcal{B}}\mathcal{U}_B$: it is clear that $\mathcal{F}\subseteq\bigcap_{B\in\mathcal{B}}\mathcal{
    U}_B$. Pick $X\in\bigcap_{B\in\mathcal{B}}\mathcal{U}_B$. Then for all $B\in\mathcal{B}$, $X\cap B$ is not empty, so $X$ is a dense subset of $\mathbb{Q}$. If $X\notin\mathcal{F}$, then there are $A\in\mathcal{F}$ and $B\in\mathcal{B}$ such that $X\cap B\cap A=\emptyset$. This implies that $X\notin \mathcal{U}_B$, which is a contradiction. Since the intersection of less than continuum many ultrafilters is a non-meager filter, then $\mathcal{F}$ is non-meager.
    
    \item [(2)] Recall that by the Talagrand-Jalaili-Naini theorem, a filter is non-meager if and only if for any interval partition $\langle I_k:k\in\omega\rangle$ of $\omega$, there is $X\in \mathcal{F}$ which has empty intersection with infinitely many intervals $I_k$. Fix a $\subseteq$-decreasing $\langle B_n:n\in\omega\rangle\subseteq\mathcal{F}$ and let $X\in\mathcal{F}$ be a pseudointersection of such sequence. Define an increasing sequence of natural numbers $\langle n_k:k\in\omega\rangle$ as $n_0=0$,  and for all $k\in\omega$, $X\setminus n_{k+1}\subseteq B_{n_k}$. Let $Y\in\mathcal{F}$ such that $Y\cap [n_k,n_{k+1})$ is empty for infinitely many $k\in\omega$. Define $Z=X\cap Y$. Let $k\in\omega$ be such that $Y\cap[n_k,n_{k+1})=\emptyset$. Then $Z\setminus n_{k}=Z\setminus n_{k+1}\subseteq B_{n_k}$.
    
    \item [(3)] Let $\langle I_k:k\in\omega\rangle$ be an interval partition of $\omega$ and $l\in\omega$ a positive natural number. For $k\in\omega$, define $J_k=\bigcup\{I_j:lk\leq j<lk+l\}$, $J^0_k=J_{2k}\cup J_{2k+1}$, and $J_k^1=J_{2k+1}\cup J_{2k+2}$. By hypothesis, we can find $X_0,X_1\in\mathcal{F}$ such that for all $k\in\omega$, $\vert J_k^0\cap X_0\vert,\vert J_k^1\cap X_1\vert\leq 1$. Let $Y=X_0\cap X_1$. Then $Y$ is as required.
\end{enumerate}

\end{proof}

The following construction was first presented in \cite{GOLDSTERN1990121} in the context of the cardinal invariants $\mathfrak{r}$ and $\mathfrak{u}$. Here we need the analogous version for $Dense(\mathbb{Q})$.

\begin{lemma}[$\mathsf{GCH}$]\label{filters_family}
There is a family $\mathscr{F}$ of selective rational filters such that $\vert\mathscr{F}\vert=\omega_2$ and for any forcing extension $W\supseteq V$ with the same cardinals, if $W\vDash\mathcal{D}\in[\mathscr{F}]^{\omega_1}$, then, in $W$, there are $\mathcal{D}'\in[\mathcal{D}]^{\omega_1}$ and an $AD$ family $\{A_\mathcal{U}:\mathcal{U}\in\mathcal{D}'\}$ such that $A_\mathcal{U}\in\mathcal{U}$ for any $\mathcal{U}\in\mathcal{D}'$.

\end{lemma}

\begin{proof}
By recursion construct a family $\{A_\eta:\eta\in\omega_1^{<\omega_1}\}$ such that:
\begin{enumerate}
    \item $A_\emptyset=\mathbb{Q}$.
    
    \item For all $\alpha\in\omega_1$, $\{A_f:f\in\omega_1^{\alpha}\}\subseteq Dense(\mathbb{Q})$ is an almost disjoint family.
    
    \item For any interval partition of $\omega$, $\langle I_n:n\in\omega\rangle$, there is $\alpha\in\omega_1$ such that for all $f\in\omega_1^\alpha$, $A_f$ is a selector of $\langle I_n:n\in\omega\rangle$.
    
    \item For all $X\in Dense(\mathbb{Q})$, there is $\alpha\in\omega_1$ such that for all $f:\alpha\to\omega_1$, either $A_f\subseteq^*X$ or $A_f\cap X\notin Dense(\mathbb{Q})$.
    
    \item For $\beta<\alpha$, $f\in \omega_1^\beta$ and $g\in\omega_1^\alpha$ such that $f\subseteq g$, we have $A_{g}\subseteq^* A_f$.
\end{enumerate}
For $f\in\omega_1^{\omega_1}$, define $\mathcal{U}_f$ to be the filter generated by the family $\{A_{f\upharpoonright\alpha}:\alpha\in\omega_1\}$, and then $\mathscr{F}=\{\mathcal{U}_f:f\in\omega_1^{\omega_1}\}$. It is easy to see that $\mathcal{U}_f$ is a selective rational filter for all $f\in\omega_1^{\omega_1}$.

Now assume $W$ is a forcing extension of $V$ having the same cardinals. Let $\{f_\alpha:\in\omega_1\}\subseteq\omega_1^{\omega_1}\cap V$ be a set of cardinality $\omega_1$. For $h\in\bigcup_{\alpha\in\omega_1}\omega_1^\alpha$, define $\langle h\rangle=\{g\in \omega_1^{\omega_1}:h\subseteq g\}$. There are two cases:

Case 1: there is $h\in \omega_1^{\omega_1}$ such that for all $\alpha\in\omega_1$, $\langle h\upharpoonright\alpha\rangle\cap\{ f_\beta:\beta\in\omega_1\}$ has cardinality $\omega_1$. Then by recursion construct $L\in[\omega_1]^{\omega_1}$ and an increasing sequence $\{\beta_\alpha\in\omega_1:\alpha\in L\}$ such that $f_{\beta_\alpha}\upharpoonright\alpha=f\upharpoonright\alpha$ and $f_{\beta_\alpha}(\alpha)\neq f(\alpha)$. Now for $\alpha\in L$ define $A_\alpha=A_{f_{\beta_\alpha}\upharpoonright(\alpha+1)}$. The families $\mathcal{A}=\{A_\alpha:\alpha\in L\}$ and $\{\mathcal{U}_{f_{\beta_\alpha}}:\alpha\in L\}$ are as required.

Case 2: for all $h\in \omega_1^{\omega_1}$ there is $\alpha\in\omega_1$ such that $\langle h\upharpoonright\alpha\rangle\cap\{f_\beta:\beta\in\omega\}$ is at most countable. Note that in this case, given $h\in\omega_1^{\omega_1}$, we can find $\alpha\in\omega_1$ such that $\langle h\upharpoonright\alpha\rangle\cap\{f_\beta:\beta\in\omega_1\}$ has cardinality at most $1$. Thus, for each $\beta\in\omega_1$, we can find $\alpha_\beta\in\omega_1$ such that $\langle f_\beta\upharpoonright\alpha_\beta\rangle\cap\{f_\gamma:\gamma\in\omega_1\}$ has exactly one element. Let $A_\beta=A_{f_\beta\upharpoonright\alpha_\beta}$. Then $\{f_\beta:\beta\in\omega_1\}$ and $\{A_\beta:\beta\in\omega_1\}$ work.

\end{proof}

\section{Preserving selective rational filters.}

In this section we study the forcing we are using and develop its main properties. This forcing was introduced in \cite{shelah_i_less_u}, and was extensively developed in \cite{osvaldo_hm} for the ultrafilter version. A close relative of this forcing was worked on \cite{chodounsky_fischer_grebik_free_sequences}. Here we work with saturated ideals. Theorem \ref{main_section_theorem} can be proved using the same argument of Shelah from \cite{shelah_i_less_u}. However, we present a different proof in the framework of games and strategies, which allows for a simpler construction in such proof. This requires to prove Proposition \ref{second_game_lemma} and Corollary \ref{third_game_lemma}. Although this is an additional technical work, we believe the resulting proof of Theorem \ref{main_section_theorem} is easier to visualize with these techniques. The proof of the Sacks property is the same from \cite{shelah_i_less_u}, which we include for completeness.

Given an ideal $\mathscr{I}$ on $\omega$, $\mathscr{I}^*=\{\omega\setminus X:X\in\mathscr{I}\}$ denotes the dual filter of $\mathscr{I}$, while $\mathscr{I}^+=\mathcal{P}(\omega)\setminus \mathscr{I}$ denotes the family of $\mathscr{I}$-positive sets.

\begin{dfn}
Let $\mathscr{I}$ be an ideal on $\omega$. An $\mathscr{I}$-partition is a family $\vec{E}\subseteq\mathscr{I}$ of non-empty subsets such that:
\begin{enumerate}
    \item $\omega\setminus\bigcup \vec{E}\in\mathscr{I}$.
    \item Every two elements from $\vec{E}$ are disjoint and non-empty.
\end{enumerate}
Given an $\mathscr{I}$-partition $\vec{E}=\langle E_n:n\in\omega\rangle$, we assume it is enumerated in the canonical way, that is, $\min(E_n)<\min(E_{n+1})$ for all $n\in\omega$. we denote the set of canonical representatives by $A(\vec{E})=\{\min(\alpha):\alpha\in E\}=\{a_n^E:n\in\omega\}$, ordered by its increasing enumeration, and $dom(\vec{E})=\bigcup_{n\in\omega}{E_n}$. For every $i\in dom(\vec{E})$, we denote by $i/\vec{E}$ the element $E_n$ of $\vec{E}$ such that $i\in E_n$.

Given a $\mathscr{I}$-partition $\vec{E}$ and $n\in\omega$, let $\vec{E}*n$ be the $\mathscr{I}$-partition defined as $\vec{E}*n=\{E_k:k\ge n\}$, with its natural order.

\end{dfn}

\begin{dfn}\label{I_partitions_ordered}
Given $\vec{E}_0,\vec{E}_1$ $\mathscr{I}$-partitions, we define $\vec{E}_1\leq \vec{E}_0$ if and only if the following holds:
\begin{enumerate}
    \item $dom(\vec{E}_1)\subseteq dom(\vec{E}_0)$.
    \item $\vec{E}_0\upharpoonright dom(\vec{E}_1)$ is a finer partition than $\vec{E}_1$, that is, elements of $\vec{E}_1$ are unions of elements of $\vec{E}_0$.
\end{enumerate}
\end{dfn}

\begin{dfn}\label{shelah_forcing}
Let $\mathscr{I}$ be an ideal on $\omega$. Define the forcing $\mathbb{Q}(\mathscr{I})$ as the set of all ordered pairs $p=(H_p,\vec{E}_p)$ such that:
\begin{enumerate}
    \item[i)] $\vec{E}_p$ is an $\mathscr{I}$-partition. We denote by $\langle E_n^p:n\in\omega\rangle$ the natural enumeration of $\vec{E}_p$. The increasing enumeration of $A(\vec{E}_p)$ will be denoted by $\langle a_n^p:n\in\omega\rangle$.
    \item[ii)] $H_p$ is a sequence of functions $H_p(n):2^{A(\vec{E}_p)\cap(n+1)}\to 2$.
    \item[iii)] For any $n\in A(\vec{E}_p)$, $H(n)=e_n$, where $e_n$ is the evaluation in the $n^{th}$-coordinate: $e_n(x)=x(n)$.
    \item[iv)] If $n=\min(i/E_p)$, then $H_p(i)=e_n$ or $H_p(i)=1-e_n$.
\end{enumerate}
The order is defined as $q\leq p$ if and only if:
\begin{enumerate}
    \item $\vec{E}_q\leq \vec{E}_p$.
    \item For all $i\in dom(\vec{E}_p)$, if $n=\min(i/\vec{E}_p)$ and $H_p(i)=H_p(n)$, then $H_q(i)=H_q(n)$; if $H_p(i)=1-H_p(n)$, then $H_q(i)=1-H_q(n)$.
    \item For all $n\in\omega\setminus dom(\vec{E}_p)$, if $x\in 2^{A(\vec{E}_q)\cap (n+1)}$, define $$x_{n,q\to p}=x\cup \langle (j,H_q(j)(x\upharpoonright(j+1))):j\in (A(\vec{E}_p)\cap (n+1))\setminus A(\vec{E}_q)\rangle$$
    Then:
    \begin{equation*}
        H_q(n)(x)=H_p(n)(x_{n,q\to p})
    \end{equation*}
\end{enumerate}
\end{dfn}

\begin{dfn}
Let $\mathscr{I}$ be an ideal on $\omega$ and $G\subseteq\mathbb{Q}(\mathscr{I})$ a generic filter over $V$. Define a $\mathbb{Q}(\mathscr{I})$-name $\dot{x}_{gen}$ for a subset of $\omega$ as follows:
\begin{equation*}
    n\in\dot{x}_{gen}\iff (\exists p\in G)(H_p(n)\equiv 1)
\end{equation*}
Here $H_{p}(n)\equiv i$ means $H_{p}(n)$ is the constant function with value $i$.
\end{dfn}

\begin{dfn}
Let $\mathscr{I}$ be a saturated ideal on $\omega$. The game $\mathscr{G}(\mathscr{I})$ between Player I and Player II play $\mathscr{I}$-partitions, is defined as follows:

\begin{center}
\begin{tabular}[t]{l |c |c |c |c |c |r}
Player I & $\vec{E}^1_0$ &      $ $      & $\vec{E}^1_1\leq \vec{E}^2_0$ & $ $    & $\ldots\ \ \ $  &\\
\hline
Player II & $ $          & $\vec{E}^2_0\leq \vec{E}^1_0$ & $ $         & $\vec{E}^2_1\leq \vec{E}^1_1$ & & $\ldots\ \ \ $\\
\end{tabular}
\end{center}

Player II wins if and only if $\bigcup_{n\in\omega}dom(\vec{E}^1_n)\setminus dom(\vec{E}^2_n)\in\mathscr{I}^*$.
\end{dfn}

Recall that an ideal $\mathscr{I}$ on $\omega$ is saturated if the quotient $\mathcal{P}(\omega)/\mathscr{I}$ is $c.c.c.$.

\begin{lemma}\label{saturated_strategy}
Let $\mathscr{I}$ be a saturated ideal. Then Player I has no winning strategy in the game $\mathscr{G(I)}$.
\end{lemma}

\begin{proof}
Assume otherwise $\Sigma$ is a winning strategy for Player I. We will produce an uncountable family of $\mathscr{I}$-positive sets. Let $\langle s_n:n\in\omega\rangle$ be the enumeration of $2^{<\omega}$ defined by the ordering $s\sqsubset t$ if and only if, either $\vert s\vert<\vert t\vert$ or $\vert s\vert=\vert t\vert$ and $s<_{lex}t$, where $<_{lex}$ is the lexicographical order ($\sqsubset$  has order type $\omega$). We construct a sequence $\langle p_s:s\in2^{<\omega}\rangle$ such that:
\begin{enumerate}
    \item $p_{\emptyset}=(\Sigma(\emptyset))=(\vec{E}_\emptyset)$.
    \item For $s\in 2^{<\omega}\setminus\{ \emptyset\}$, $p_s=(\vec{F}_s,\vec{E}_s)$.
    \item $\vec{F}_{s_1}=\vec{E}_\emptyset$ and $\vec{E}_{s_1}=\Sigma(\vec{E}_{\emptyset},\vec{F}_{s_1})$.
    \item For $n>1$:
    \begin{enumerate}
        \item If $\vec{E}_{s_n}$ is defined, then define $\vec{F}_{s_{n+1}}=\vec{E}_{s_n}$.
        \item Now, let 
        \begin{equation*}
            P(s_{n+1})=(\vec{E}_\emptyset,\vec{F}_{s_n\upharpoonright 1},\vec{E}_{s_n\upharpoonright 1},\ldots,\vec{F}_{s_n\upharpoonright l},\vec{E}_{s_n\upharpoonright l},\ldots,\vec{E}_{s_{n+1}\upharpoonright(\vert s_{n+1}\vert-1)}, \vec{F}_{s_{n+1}})
        \end{equation*}
        Then make $\vec{E}_{s_{n+1}}=\Sigma(P(s_{n+1}))$.
    \end{enumerate}
    
\end{enumerate}

Now, $f\in 2^\omega$, define $br(f)$ as
\begin{equation*}
    br(f)=(\vec{E}_\emptyset,\vec{F}_{f\upharpoonright 1},\vec{E}_{f\upharpoonright 1},\ldots,\vec{F}_{f\upharpoonright l},\vec{E}_{f\upharpoonright l},\ldots)_{l\in\omega}
\end{equation*}
By the construction, $br(f)$ is a run of the game $\mathscr{G}(\mathscr{I})$ on which Player I followed the strategy $\Sigma$, so we have that $D(f)=\bigcup_{n\ge 1}dom(\vec{F}_{f\upharpoonright n})\setminus dom(\vec{E}_{f\upharpoonright n})\in\mathscr{I}^+$. However, for different $f,g\in 2^{\omega}$ we have that $D(f)\cap D(g)\in\mathscr{I}$, which implies that $\{D(f):f\in 2^{\omega}\}$ is an uncountable antichain in $\mathscr{I}^+$, which is a contradiction.
\end{proof}

\noindent\textbf{Notation.} Given a condition $p\in\mathbb{Q}(\mathscr{I})$, and $s\in 2^n$, define $p^s$ as follows:
\begin{enumerate}
    \item $dom(\vec{E}_{p^s})=\vec{E}_p*\vert s\vert$.
    \item $\vec{E}_{p^s}=\vec{E}_p\upharpoonright dom(\vec{E}_{p^s})$.
    \item For $k<\vert s\vert$, $H_{p^s}(a_k^p)\equiv s(k)$.
    \item If $j\in a_i^p/\vec{E}_p$ for some $i<\vert s\vert$, we have that 
    \begin{enumerate}
        \item If $H_{p}(j)=H_p(a_i^p)$, then $H_{p^s}(j)\equiv s(i)$.
        \item If $H_{p}(j)=1-H_p(a_i^p)$, then $H_{p^s}(j)= 1-s(i)$.
    \end{enumerate}
    \item For $j\in dom(\vec{E}_{p^s})$, and $x\in 2^{A(\vec{E}_{p^s})\cap(j+1)}$, note that $$x_{j,p^s\to p}=\langle(a_i^p,s(i)):i<n\rangle\cup x$$
    and define
    $$H_{p^s}(j)(x)=H_p(j)(x_{j,p^s\to p})$$
    
    \item If $j\notin dom(\vec{E}_p)$, then for all $x\in 2^{A(\vec{E}_{p^s})\cap (j+1)}$ we have $H_{p^s}(j)(x)=H_{p}(j)(x_{j,p^s\to p})$.
\end{enumerate}

\begin{dfn}
Let $\mathscr{I}$ be an ideal on $\omega$, $p,q\in\mathbb{Q}(\mathscr{I})$ and $n\in\omega$. We define $q\leq_n p$ if and only if the following happens:
\begin{enumerate}
    \item $q\leq p$.
    \item For all $i\leq n$, $a_i^p=a_i^q$.
    \item For all $i\leq n$, $a_i^p/\vec{E}_p=a_i^q/\vec{E}_q$.
\end{enumerate}
\end{dfn}

The following fact is easy to prove.

\begin{lemma}
Let $\mathscr{I}$ be a saturated ideal. Then: $\mathbb{Q}(\mathscr{I})\Vdash\dot{x}_{gen},\omega\setminus\dot{x}_{gen}\in\mathscr{I}^+$.
\end{lemma}

\begin{lemma}\label{initial_guessing}
Let $\mathscr{I}$ be an ideal, $p\in\mathbb{Q}(\mathscr{I})$ a condition, $\dot{x}$ a $\mathbb{Q}(\mathscr{I})$-name for an element in $V$ and $n\in\omega$. Then there is $q\leq_{n-1} p$ such that for all $s\in 2^n$, $q^s$ decides the value of $\dot{x}$. Moreover, the set of conditions $\{q^{s}:s\in 2^n\}$ is a maximal antichain below $q$.
\end{lemma}

\begin{proof}
Let $\{s_i:i< 2^n\}$ be an enumeration of all the sequences $s:n\to 2$. Define a decreasing sequence of conditions $\langle q_i:i\in\{-1\}\cup 2^n\rangle$ such that:
\begin{enumerate}
    \item $q_{-1}=p$.
    \item $q_{i+1}\leq_{n-1} q_i$.
    \item $q_{i}^{s_i}$ decides the value of $\dot{x}$.
\end{enumerate}
The construction is done as follows:
\begin{enumerate}
    \item Assume $q_i$ has been defined. Find $r_{i+1}\leq q_i^{s_{i+1}}$ which decides the value of $\dot{x}$. Then define $q_{i+1}$ as:
    \begin{enumerate}
        \item $dom(q_{i+1})=\bigcup_{j<n}a_{j}^{p}/\vec{E}_p\cup dom(\vec{E}_{r_{i+1}})$.
        
        \item $\vec{E}_{q_{i+1}}=\{a_j^p/\vec{E}_p:j<n\}\cup \vec{E}_{r_{i+1}}$.
        
        \item If $m\in a_{j}^p/\vec{E}_{q_{i+1}}$ for some $j<n$, then $H_{q_{i+1}}(j)=H_{p}(j)$.
        
        \item For all $m\in dom(\vec{E}_{r_{i+1}})$ and $x\in 2^{A(\vec{E}_{q_{i+1}})\cap (m+1)}$, $$H_{q_{i+1}}(m)(x)=H_{r_{i+1}}(m)(x\upharpoonright A(\vec{E}_{r_{i+1}}))$$
        
        \item For $m\notin dom(\vec{E}_{q_i})$ and $x\in 2^{A(\vec{E}_{q_{i+1}})\cap(m+1)}$, $$H_{q_{i+1}}(m)(x)=H_{q_i}(m)(x_{m,q_{i+1}\to q_i})$$
    \end{enumerate}
    It follows directly from the definitions that $q_{i+1}\leq_{n-1} q_i$.
\end{enumerate}
Finally, note that by construction, for all $i<2^n$ we have $q_{2^n-1}^{s_i}\leq q_{i}^{s_i}=r_{i}$, and $r_{i}$ decides the value of $\dot{x}$, so $q_{2^n-1}^{s_i}$ decides the value of $\dot{x}$.
\end{proof}

\begin{dfn}
Let ${p}\in\mathbb{Q}(\mathscr{I})$ be a condition, $n\in\omega$ and $\vec{F}$ an $\mathscr{I}$-partition such that $\vec{F}\leq \vec{E}_p*n$ and $dom(\vec{F})=dom(\vec{E}_p*n)$. Define the condition $p*\vec{F}$ as follows:
\begin{enumerate}
    \item $\vec{E}_{p*\vec{F}}=\{E_i^p:i<n\}\cup \vec{F}$. Note that $dom(\vec{E}_{p*\vec{F}})=dom(\vec{E})$.
    \item If $a\in A(\vec{E}_p)\setminus A(\vec{E}_{p*\vec{E}})$, define $H_{p*\vec{E}}(a)=H_{p*\vec{E}}(\min(a/\vec{E}_{p*\vec{F}}))$.
    \item For $i\in dom(\vec{E}_{p*\vec{F}})$:
    \begin{enumerate}
        \item If $H_p(i)=H_p(\min(i/\vec{E}_p))$, then $H_{p*\vec{F}}(i)=H_{p*\vec{F}}(\min(i/\vec{E}_p))$.
        \item If $H_p(i)=1-H_p(\min(i/\vec{E}_p))$, then $H_{p*\vec{F}}(i)=1-H_{p*\vec{F}}(\min(i/\vec{E}_p))$.
    \end{enumerate}
    \item For $i\notin dom(\vec{E}_p)$, and $x\in 2^{A(p*\vec{F})\cap(i+1)}$, define $$H_{p*\vec{E}}(i)(x)=H_p(i)(x_{i,p*\vec{E}\to p})$$ 
\end{enumerate}
Note that $p*\vec{F}\leq_{n-1} p$.
\end{dfn}

\begin{prp}\label{sacks_property}
If $\mathscr{I}$ be a saturated ideal, then $\mathbb{Q}(\mathscr{I})$ has the Sacks property.
\end{prp}

\begin{proof}
Let $\dot{f}$ be a $\mathbb{Q}(\mathscr{I})$-name and $p\in\mathbb{Q}(\mathscr{I})$ a condition. We construct a  decreasing sequence $\langle q_n:n\in\omega\rangle$, and on the side give a strategy to Player I in the game $\mathscr{G(I)}$. Since this is not a winning strategy, there is a run of the game in which Player II wins. The existence of such run of the game will give us the desired condition.

The construction of the sequence is done as follows:
\begin{enumerate}
    \item Start by choosing $q_0\leq p$ which defines $\dot{f}(0)$; then Player I plays the $\mathscr{I}$-partition $\vec{E}_{q_0}*1$
    \item Suppose $q_n$ is constructed and Player I has played $\vec{E}^1_n=\vec{E}_{q_n}*(n+1)$. Then Player II has answered with $\vec{E}^2_n$. Define $\vec{G}_n=\{dom(\vec{E}_n^1)\setminus dom(\vec{E}_n^2)\}\cup \vec{E}_n^2$. Clearly $\vec{G}_n$ is an $\mathscr{I}$-partition. Then find $q_{n+1}\leq_{n} q_n*\vec{G}_n$ and such that for all $s:n+1\to 2$, $q_{n+1}^s$ decides the value of $\dot{f}(n+1)$. Then Player I plays $\vec{E}^1_{n+1}=\vec{E}_{q_{n+1}}*(n+2)$.
    \end{enumerate}

Since this is not a winning strategy for Player I, there is a run of the game in which $\bigcup_{n\in\omega}dom(\vec{E}_n^1)\setminus dom(\vec{E}_n^2)\in\mathscr{I}^*$. Suppose $\langle q_n:n\in\omega\rangle$ is the sequence constructed along such run of the game. Let $F_n=dom(E_n^1)\setminus dom(E_n^2)$ and define $\vec{F}=\langle F_n:n\in\omega\rangle$. Clearly, $\vec{F}$ is a  $\mathscr{I}$-partition. Note that for any $n\in\omega$, $F_n=E_{n}^{q_{n+1}}$, and actually, for any $k\ge n+1$ we have that $F_n=E_{n}^{q_{k+1}}$. Also note that for any $n\in\omega$ and $x\in 2^{A(\vec{F})\cap(n+1)}$, the sequence $\langle H_{q_k}(n)(x):k\in\omega\rangle$ is eventually constant (it may be not defined on an initial segment, but at some point it gets well defined, and even eventually constant).

Define a condition $q_\omega$ as follows:
\begin{enumerate}
    \item[(I)] $\vec{E}_{q_\omega}=\vec{F}$.
    \item[(II)] For all $n\in\omega$ and $x\in 2^{A(E_{q_\omega})\cap (n+1)}$, $H_{q_\omega}(n)(x)=\lim_{k\in\omega} H_{q_k}(n)(x)$.
\end{enumerate}

We need to prove that $q_\omega$ is a condition in $\mathbb{Q}(\mathscr{I})$ and a lower bound of the sequence $\langle q_n:n\in\omega\rangle$. This clearly is sufficient to finish the proof.

Clauses i) and ii) from Definition \ref{shelah_forcing} are immediate (clause ii) holds because the limit of clause (II) above is well defined). For clause iii), pick $a_k^{q_\omega}\in A(E_{q_\omega})$, note that $a_k^{q_\omega}=a_{k}^{q_{k+1}}$. Moreover, for all $l\ge k+1$, and any $j\leq k+1$ $a_j^{q_\omega}/E_{q_\omega}=a_{j}^{q_{\omega}}/E_{q_{l}}$. This in turn implies that for $l\ge k+1$, $dom(H_{q_l}(a_k^{q_\omega}))=dom(H_{q_{k+1}}(a_k^{q_\omega}))=2^{A(E_{q_\omega})\cap (a_{k}^{q_\omega}+1)}$. From this it follows that for $l\ge k+1$,
\begin{equation*}
    H_{q_l}(a_{k}^{q_\omega})(x)=e_{{a_k^{q_{k+1}}}}(x)=x(a_{k}^{q_{k+1}})
\end{equation*}
By definition (II) of $H_{q_\omega}(a_k^{q_\omega})$, we have $H_{q_\omega}(a_k^{q_\omega})=x(a_k^{q_{\omega}})=e_{a_k}^{q_\omega}(x)$. It is also easy to see that this implies that clause iv) from Definition \ref{shelah_forcing} is also satisfied.

Fix a condition $q_n$, we want to prove that $q_\omega\leq q_n$. Note that $\vec{E}_{q_\omega}\leq \vec{E}_{q_n}$ follows the definition of $\vec{E}_{q_\omega}$. To check clause (2) from Definition \ref{shelah_forcing}, pick $l\in dom(E_{q_\omega})$ and note that for big enough $k\ge \max(n+1,l+1)$, $A(E_{q_\omega})\cap (l+1)=A(E_{q_k})\cap (l+1)$, which implies that for big enough $k\ge \max(n+1,l+1)$
\begin{equation*}
    H_{q_\omega}(l)(x)=H_{q_k}(l)(x)
\end{equation*}
It is not hard to see that this implies that the the complete requirement of clause (2) holds. To check that clause (3) from Definition \ref{shelah_forcing} holds, pick $l\notin dom(\vec{E}_{q_n})$, and note in addition that for big enough $k\ge\max(n+1,l+1)$, it holds that 
\begin{equation*}
 H_{q_\omega}(l)(x)=H_{q_k}(l)(x)=H_{q_n}(l)(x_{l,q_k\to q_n}) 
\end{equation*}
where the last inequality follows from $q_l\leq q_n$. Therefore, $$H_{q_\omega}(l)(x)=H_{q_n}(l)(x_{l,q_{\omega}\to q_n})$$ since $x_{l,q_{\omega}\to q_n}=x_{l,q_{k}\to q_n}$ for big enough $k\in\omega$. 

Finally, note that for each $s\in 2^n$, $q_\omega^s\leq q_n^s$ decides a value for $\dot{f}(n)$, so $q_\omega$ bounds the possible values of $\dot{f}(n)$ to a set of size at most $2^n$.
\end{proof}

\begin{prp}
If $\mathscr{I}$ is a saturated ideal on $\omega$, then $\mathbb{Q}(\mathscr{I})$ is proper.
\end{prp}

\begin{proof}
The argument completely follows similar lines to the proof of Proposition \ref{sacks_property}, with small modifica\-tions. Let $\mathcal{M}\prec H(\theta)$ be a countable elementary submodel such that $\mathbb{Q}(\mathscr{I})\in\mathcal{M}$, and pick a condition $p\in\mathcal{M}\cap\mathbb{Q}(\mathscr{I})$. Let $\langle\dot{\alpha}_n:n\in\omega\rangle$ be an enumeration of all the $\mathbb{Q}(\mathscr{I})$-names in $\mathcal{M}$ for ordinals. Then apply the construction of Lemma \ref{sacks_property}, with the modifica\-tion that instead of guessing $\dot{f}(n)$ at step $n$, we are guessing the values of the name $\dot{\alpha}_n$, plus the fact that all the steps of the recursion can be done inside $\mathcal{M}$, that is, the $\mathscr{I}$-partitions played by both players are in $\mathcal{M}$; the same proof of Lemma \ref{saturated_strategy} shows that in this variation of the game $\mathscr{G}(\mathscr{I})$ Player I has no winning strategy (the only modification is that now the strategy $\sigma$ shoots $\mathscr{I}$-partitions living in $\mathcal{M}$), so we can get the condition $q_\omega$ of Proposition \ref{sacks_property}. Then the condition $q_\omega$ is an $(\mathcal{M},\mathbb{Q}(\mathscr{I}))$-generic condition.
\end{proof}

\begin{dfn}
Let $\mathcal{F}$ and $\mathcal{G}$ be filters on $\omega$. We say that $\mathcal{F}$ and $\mathcal{G}$ are nearly coherent if there is a finite to one function $f:\omega\to\omega$ such that $f(\mathcal{F})\cup f(\mathcal{G})$ generates a filter.
\end{dfn}

Note that if we have that $\mathcal{F}$ and $\mathcal{G}$ are not nearly coherent, then both filters are non-meager.

\begin{dfn}
Let $\mathcal{F}$ and $\mathcal{G}$ be filters on $\omega$. The game $\mathscr{G}(\mathcal{F},\mathcal{G})$ between Player I and Player II, is defined as follows: a complete round of the game consist of four movements, defined as... at round number $n$, both players follow the next pattern:

\begin{enumerate}
    \item After Player II has answered with a natural number $k_{n-1}$ (if $n=0$, we make $k_{-1}=-1$), Player I plays a natural number $m^0_n>m^0_{n-1}$.
    \item Then Player II answers with a natural number $m_n^1>m^0_{n}$.
    \item Then Player I answers with a set $A_n\in\mathcal{G}$.
    \item Finally, Player II answers with a natural number $k_n\in A_n$.
\end{enumerate}
Player II wins if and only if $$\bigcup_{n\in\omega}[m_n^0,m_n^1)\in\mathcal{F}\text{      and      }\{k_n:n\in\omega\}\in\mathcal{G}$$
\end{dfn}

We need the following lemma from  \cite{todd}

\begin{lemma}[Eisworth, see \cite{todd}]\label{todd_lemma}
Let $\mathcal{F},\mathcal{G}$ be filters which are not nearly coherent. Let $\mathcal{I}=\{I_n\colon n\in\omega\}$ be an interval partition of $\omega$. Then there is an interval partition $\mathcal{J}=\{J_n\colon n\in\omega\}$ such that each
	 $J_n$ is a union of elements of $\mathcal{I}$, and
\begin{equation*}
    \bigcup_{k\in\omega} J_{4k}\in\mathcal{F}\text{   and   }    \bigcup_{k\in\omega} J_{4k+2}\in\mathcal{G}
\end{equation*}
\end{lemma}

\begin{prp}\label{second_game_lemma}
Let $\mathcal{F}$ and $\mathcal{G}$ be two filters which are not nearly coherent. Further, assume that $\mathcal{G}$ is a selective filter. Then Player I has no winning strategy in the game $\mathscr{G}(\mathcal{F},\mathcal{G})$.    
\end{prp}

\begin{proof}
Let $\Sigma$ be a strategy for Player I, we need to produced a run of the game in which Player I follows $\Sigma$ but Player II wins. We think of $\Sigma$ as the family sequences of possible and allowed moves of both players, so we identify $\Sigma$ with a tree with the following properties: for $s\in \Sigma$ and $i<\vert s\vert$
\begin{enumerate}
    \item It has a root node $\langle m_0^0\rangle$, which is the starting move of Player I.
    \item If $i=4N$ for some $N$, then $s(i)\in\omega$ is the movement of Player I determined by $\Sigma$. The successors of $s\upharpoonright(i+1)$ is the set of all the possibilities that Player II is allowed to play, which is the set of all natural numbers bigger than $s(i)$.
    
    \item If $i=4N+2$ for some $N$, then $s(i)\in\mathcal{G}$ is the movement of Player $I$ determined by $\Sigma$. The successors of $s\upharpoonright(i+1)$ is the set of all the possibilities that Player II is allowed to play, which is the set of all natural numbers which are elements of $s(i)$.
\end{enumerate}
For $n\in\omega$, define $\Sigma^{n}=\{s\in\Sigma:\vert s\vert=n+1\}$. For $j\in 4$, let $\Sigma^{[j]}=\bigcup_{k\in\omega}\Sigma^{4k+j}$
Also, if $s\in\Sigma$ is a state of the game in which last movement was given by Player II, $\Sigma(s)$ denotes the answer of Player I to the last movement of Player II in $s$, according to $\Sigma$.

For $n\in\omega$, defined the following sets

$$\mathcal{E}_n=\{s\in\Sigma^{[2]}:(\forall i<\vert s\vert)(\text{if } 4i+1<\vert s\vert\Longrightarrow s(4i+1)\leq n)\}$$

Each one of this sets is finite. Now, since $\Sigma$ is a countable set, we have that the family of sets $A\in\mathcal{G}$ which appear in some $s\in \Sigma$ is countable, so we can find $A_\omega\in\mathcal{G}$ which is a pseudointersection of all of them. We construct two sequences of natural numbers $\langle l_n:n\in\omega\rangle$ and $\langle \alpha_n:n\ge 2\rangle$ as follows:
\begin{enumerate}
    \item $l_0=\Sigma(\emptyset)=m_0^0$ and $l_1>l_0$ is such that $A_\omega\cap [l_0,l_1)\neq\emptyset$.
    \item For each $j\in (l_0,l_1]$, $\langle l_0,j\rangle$ is an allowed state of the game in which last movement was given by Player II, so $\Sigma(\langle l_0,j\rangle)=A_j$ is well defined. Choose $\alpha_2>l_1$ such that for each $j\in(l_0,l_1]$, $A_j\cap \alpha_2\neq\emptyset$. Now pick $l_2>\alpha_2$ such that for each $j\in(l_0,l_1]$ and $i\in A_j\cap \alpha_2$, $\Sigma(l_0,j,A_j,i)\leq l_2$.
    \item If $l_n\ge l_2$ has been defined, define $l_{n+1}>\alpha_{n+1}>l_n$ such that the following holds:
    \begin{enumerate}    
        \item $\alpha_{n+1}>l_n$ is such that for each $s\in\mathcal{E}_{l_n}$ and $i<\vert s\vert$ such that $4i+2<\vert s\vert$, there is $j\in s(4i+2)\cap[\alpha_n,\alpha_{n+1})$. 
        
        \item We require $l_{n+1}>\alpha_{n+1}$ to satisfy that for each $s\in\mathcal{E}_{l_{n}}$ and $i<\vert s\vert$ such that $4i+2<\vert s\vert$, and $j\in s(4i+2)\cap\alpha_{n+1}$, $\Sigma(s\upharpoonright(4i+3)^\frown j)<l_{n+1}$.
        
        \item Define $H_{n+1}=\{s(4i+2):s\in \mathcal{E}_{l_{n}}\land 4i+2<\vert s\vert\}$. Then, for all $A\in H_{n+1}$, $A\setminus l_{n+1}\subseteq A_{\omega}$.
        
        \item $A_\omega\cap[l_n,l_{n+1})\neq\emptyset$.
    \end{enumerate}
\end{enumerate}

Then we get an interval partition $\langle I_n=[l_n,l_{n+1}):n\in\omega\rangle$ and by Lemma \ref{todd_lemma}, we can find and interval partition $\mathcal{J}$ such that each element of $\mathcal{J}$ is union of intervals $I_n$, and,
$$\bigcup_{k\in\omega}J_{4k}\in\mathcal{F}\text{      and      }\bigcup_{k\in\omega}J_{4k+2}\in\mathcal{G}$$ Let $B\in\mathcal{G}$ be a selector for $\mathcal{J}$ and define $B_\omega=B\cap A_\omega\cap\bigcup_{k\in\omega}J_{4k+2}$ and let $\{b_n:n\in\omega\}$ be its increasing enumeration. Note that since we assume $A_\omega\cap[l_n,l_{n+1})\neq\emptyset$ for all $n\in\omega$, we can assume that $B_\omega\cap J_{4n+2}\neq\emptyset$ for all $n\in\omega$. Let also $\langle \eta_k:k\in\omega\rangle$ be the increasing sequence such that $J_k=[\eta_k,\eta_{k+1})$, and $\langle \epsilon_k:k\in\omega\rangle$ be such that $\eta_k=l_{\epsilon_k}$. Now we defined the following run of the game:
\begin{enumerate}
    \item Player I starts by playing $m_0^0=l_0$ and let Player II play $m_0^1=\eta_1=l_{\epsilon_1}$. Note that $[m_0^0,m_0^1)=J_0$. Define $A_0=\Sigma(m_0^0,m_0^1)$, so $A_0\in H_{{\epsilon_1}+1}$ (since ${m_0^0}^\frown{m_1^0}^\frown A_0\in\mathcal{E}_{l_{\epsilon_1}}$), which implies that $A_\omega\setminus l_{\epsilon_1+1}\subseteq A_0$. Note that $b_0\ge \min(J_2)=l_{\epsilon_2}\ge l_{\epsilon_1+1}$, so Player II can play $b_0$. This finishes the first round of the game.
    \item Assume we have constructed a partial run of the game $s$ such that $\vert s\vert=4n$ for some positive $n\in\omega$. Then the last movement was given by Player II and corresponds to Player II playing $b_{n-1}\in J_{4(n-1)+2}$. Then we have $b_{n-1}\in [l_{j},l_{j+1})$ for some $j+1\leq \epsilon_{4(n-1)+3}$, which implies that $b_{n-1}<l_{j+1}<\alpha_{j+2}$, which means that $\Sigma(s)=m_{n}^0<l_{j+2}\leq l_{\epsilon_{4(n-1)+4}}=l_{\epsilon_{4n}}=\eta_{4n}<\eta_{4n+1}$, and Player II can play $m_n^1=\eta_{4n+1}$. Define $A_n=\Sigma(s^\frown {m_n^0}^\frown m_n^1)$, and note that $A_n\in H_{\epsilon_{4n+1}+1}$ (since $s^\frown {m_n^0}^\frown {m_n^1}^\frown A_n\in\mathcal{E}_{l_{\epsilon_{4n+1}}}$), which implies that $A_\omega\setminus l_{\epsilon_{4n+1}+1}\subseteq A_n$. Since $b_n\in J_{4n+2}\cap B_\omega$ and $\min(J_{4n+2})=\eta_{4n+2}=l_{\epsilon_{4n+2}}\ge l_{\epsilon_{4n+1}+1}$, we have that $b_n\in A_n$, so Player II can play $b_n\in A_n$. This finishes the definition of the current round, which results in $s^\frown {m_n^0}^\frown{m_n^1}^\frown {A_n}^\frown b_n$. Note that $J_{4n}\subseteq[m_n^0,m_n^1)$.
\end{enumerate}
Note that after this run of the game, $\bigcup_{k\in\omega}J_{4k}\subseteq \bigcup_{k\in\omega}[m_k^0,m_k^1)$ and the second set constructed by Player II is exactly $B_\omega$. Thus, Player II wins.
\end{proof}

\begin{dfn}
Let $\mathcal{F}$ and $\mathcal{G}$ be two filters on $\omega$. Let $\vec{E}$ be an $\mathcal{F}^*$-partition. Define the game $\mathscr{G}_{\vec{E}}(\mathcal{F},\mathcal{G})$ as the game $\mathscr{G}(\mathcal{F},\mathcal{G})$, with the exception that Player II wins if and only if $$\bigcup_{n\in\omega}\bigcup_{j\in[m_n^0,m_n^1)}E_j\in\mathcal{F}\text{         and         }\{k_j:j\in\omega\}\in\mathcal{G}$$
\end{dfn}

\begin{crl}\label{third_game_lemma}
Assume $\mathcal{F}$ and $\mathcal{G}$ filters. Further, assume $\mathcal{G}$ is a selective filter, and let $\vec{E}$ be an $\mathcal{F}^*$-partition. Define a function ${h}: dom(\vec{E})\to\omega$ as $h(k)=n$ if and only if $k\in E_n$. Assume $h(\mathcal{F})$ and $\mathcal{G}$ are not nearly coherent. Then Player I has no winning strategy in the game $\mathscr{G}_{\vec{E}}(\mathcal{F},\mathcal{G})$.
\end{crl}

\begin{proof}
If $\Sigma$ is a winning strategy for Player I in the game $\mathscr{G}_{\vec{E}}(\mathcal{F},\mathcal{G})$, then $\Sigma$ is a winning strategy in the game $\mathscr{G}(h(\mathcal{F}),\mathcal{G})$. However, $h(\mathcal{F})$ and $\mathcal{G}$ are not nearly coherent.
\end{proof}

\begin{dfn}
Let $\mathscr{I}$ be an ideal on $\omega$. We say that a function $f:\omega\to\omega$ is $\mathscr{I}$-to-one if and only if for any $n\in\omega$, $f^{-1}[\{n\}]\in\mathscr{I}$.
\end{dfn}

Theorem \ref{main_section_theorem} below is the main theorem of this section.

\begin{thm}\label{main_section_theorem}
Let $\mathscr{I}$ be a saturated ideal on $\omega$ and let $\mathcal{F}$ be a selective rational filter such that for no $\mathscr{I}$-to-one function
$f$, $\mathcal{F}\cup f(\mathscr{I}^*)$ generates a filter. Then $\mathbb{Q}(\mathscr{I})$ forces that $\mathcal{F}$ generates a 
selective rational filter in the generic extension.
\end{thm}

\begin{proof}
Let $\dot{X}$ be a name for a dense subset of the rational numbers and $p\in\mathbb{Q}_\mathscr{I}$ a condition. We have to find a condition $q\leq p$ such that
either, there is  $Z\in\mathcal{F}$ such that $q\Vdash Z\cap \dot{X}\notin Dense(\mathbb{Q})$, or $q\Vdash Z\subseteq \dot{X}$. By Proposition \ref{sacks_property} we can assume that for all $n\in\omega$ and $s\in 2^n$, $p^s$ decides the truth value of $n\in\dot{X}$.   

For each $s\in 2^{<\omega}$, let $X_s$ be the set $X_s=\{n\in\omega:\exists q\leq p^{s}\ q\Vdash n\in\dot{X}\}$. If there is $s\in 2^{<\omega}$
such that $X_s\notin\mathcal{F}$, then there is $Z\in\mathcal{F}$ such that $X_s\cap Z\notin Dense(\mathbb{Q})$. Since $p^{s}$ forces $\dot{X}$ to be a subset of $X_s$, then $\dot{X}\cap Z$ is forced by $p^s$ to be a not dense subset of the rational numbers, so we are done with this case and we can assume that $X_s\in\mathcal{F}$ for all $s\in 2^{<\omega}$. 

Define a function $f:dom(\vec{E}_p)\to\omega$ as $f(k)=n$ if and only if $k\in E_n$. Then $\mathcal{F}\cup f(\mathscr{I}^*)$ does not generates a filter. Moreover, our assumptions on $\mathscr{I}^*$ and $\mathcal{F}$ imply that $\mathcal{F}$ and $f(\mathscr{I}^*)$ are not nearly coherent, so Player I has no winning strategy in the game $\mathscr{G}_{\vec{E}_p}(\mathscr{I}^*,\mathcal{F})$. We define a strategy for Player I in the game $\mathscr{G}_{\vec{E}_p}(\mathscr{I}^*,\mathcal{F})$, in which Player I constructs along the game a decreasing sequence of conditions. This is done as follows:
\begin{enumerate}
    \item Player I starts by playing $m_0^0=0$. Define $q_{-1}=p$.
    \item After Player answers with $m^1_0>m_0^0$, define $A_0=\bigcap_{s\in 2^{m_0^1}}X_s$, and let Player I play the set $A_0$. Then Player II answers with $k_0\in A_0$. This finishes the first round of the game.
    \item Assume round number $n$ has finished and Player II has played $k_n\in A_n$. Then, for each $s\in 2^{m_n^1}$, pick $r_s\in 2^{<\omega}$ extending $s$ and such that $p^{r_s}\Vdash k_n\in A_n$. Let $m_{n+1}^0\ge\max\{\vert r_s\vert:s\in 2^{m_n^1}\}$ and let Player I play $m_{n+1}^0$. Then Player II answers with $m_{n+1}^1>m_{n+1}^0$. Let Player I play the set $A_{n+1}=\bigcap_{s\in 2^{m_{n+1}^1}}X_s$. Then Player II answers with $k_{n+1}\in A_{n+1}$, and the round is finished.
    
    Now we define condition $q_n$. For $s\in 2^{m_n^1}$, pick $\circ s\in 2^{m_{n+1}^0}$ such that $r_s\subseteq \circ s$. Define $q_{n}$ as follows:
    \begin{enumerate}
        \item $dom(\vec{E}_{q_n})=\bigcup_{l\leq n}\bigcup_{j\in[m_l^0,m_l^1)}E^{q_{n-1}}_j\cup\bigcup_{j\ge m_{n+1}^0}E_j^{q_{n-1}}$.
        \item $\vec{E}_{q_n}=\vec{E}_{q_{n-1}}\upharpoonright dom(\vec{E}_{q_n})=\vec{E}_p\upharpoonright dom(\vec{E}_{q_n})$.
        \item For $j\in[m_{n}^1,m_{n+1}^0)$, and $x\in 2^{A(E_{q_n})\cap (a_{j}^{q_{n-1}}+1)}$, let $s_x=x\upharpoonright a_{m_n^1}^{q_n}$ and define $H_{q_n}(a_{j}^{q_{n-1}})(x)=\circ s_x(j)$. This implicitly defines $H_{q_\omega}(j)$ for $j\in\bigcup_{l\in[m_{n}^1,m_{n+1}^0)}E_l^{q_{n-1}}$ (which should satisfy clause (2) from Definition \ref{shelah_forcing}).
        \item For $j\in dom(\vec{E}_{q_n})$, define $H_{q_n}(j)$ in the natural way to satisfy clauses ii) to iv) and (2) from Definition \ref{shelah_forcing}.
        \item For $j\notin dom(\vec{E}_{q_{n-1}})$, $H_{q_n}(j)$ is defined accordingly to clause (3) from Definition \ref{shelah_forcing}.
    \end{enumerate}
    Note that for each $s\in 2^{m_n^1}$, $q_n^s\leq p^{\circ s}\leq p^{r_s}$, so $q_n^s\Vdash k_n\in\dot{X}$. Thus, $q_n\Vdash k_n\in\dot{X}$.
\end{enumerate}
This is not a winning strategy for Player I, so there is a run of the game in which Player II wins, so, for such run, we have $\bigcup_{n\in\omega}\bigcup_{j\in[m_n^0,m_n^1)}E^p_j\in\mathscr{I}^*$ and $\{k_n:n\in\omega\}\in\mathcal{F}$. Let $\langle m_n^0,m_n^1,k_n, q_n:n\in\omega\rangle$ be the sequences constructed along such run of the game. Define a condition $q_\omega$ as follows:
\begin{enumerate}
    \item $dom(\vec{E}_{q_\omega})=\bigcup_{n\in\omega}\bigcup_{j\in [m_n^0,m_n^1)}E^p_{j}\in\mathscr{I}^*$.
    \item $\vec{E}_{q_\omega}=\vec{E}_p\upharpoonright dom(\vec{E}_{q_\omega})$.
    \item For any $j\in\omega$ and $x\in 2^{A(\vec{E}_{q_\omega})\cap (j+1)}$, let $H_{q_\omega}(j)=\lim_{n\to\omega} H_{q_n}(j)(x)$. 
\end{enumerate}
Checking that $q_\omega$ is a condition and a lower bound of $\{q_n:n\in\omega\}$ is quite similar to the case of Proposition \ref{sacks_property}, so we omit the details. Finally, note that since $q_n\Vdash k_n\in \dot{X}$ and $q_\omega\leq q_n$ for all $n\in\omega$, we also have that $q_\omega\Vdash \{k_n:n\in\omega\}\subseteq \dot{X}$.
\end{proof}

\section{Killing irresolvable spaces.}

Before proving the main theorem of the paper, we need to state a few lemmas, for which we include the proofs for self-contained purposes. Recall from the Introduction that a topological space is strongly irresolvable if any open subset is irresolvable with the subspace topology.

\begin{thm}\label{main_theorem}
It is relatively consistent with $\mathsf{ZFC}$ that $\omega_1=\mathfrak{r}_{\mathsf{scattered}}=\mathfrak{r}_{\mathsf{nwd}}<\mathfrak{irr}=\mathfrak{i}=\omega_2$.
\end{thm}

\begin{lemma}\label{strongly_irr_subspace}
Let $(\omega,\tau)$ be an irresolvable space. Then there is an open $X\subseteq\omega$ such that $(X,\tau\upharpoonright X)$ is strongly irresolvable, that is, every $\tau\upharpoonright X$-open subset of $X$ is irresolvable.
\end{lemma}

\begin{proof}
Let $R$ be the collection of all open subsets of $\omega$ which are resolvable. Note that if $U\in R$ and $V\subseteq U$ is open, then $V\in R$. Let $\mathcal{A}\subseteq R$ be a maximal subset of disjoint open sets. Note that $W=\bigcup \mathcal{A}$ is resolvable, and $\omega\setminus W$ has non-empty interior(otherwise $(\omega,\tau)$ would be resolvable). Define $X=int_{\tau}(\omega\setminus W)$. Clearly $(X,\tau\upharpoonright X)$ is strongly irresolvable: otherwise, there would be an open set $U\subseteq X$ which is resolvable, and since $X$ is open in $\tau$, $U$ is open in $\tau$, so we get $U\in\mathcal{R}$, which implies $X\cap W\neq\emptyset$, which is a contradiction.
\end{proof}

\begin{crl}
\begin{equation*}
    \mathfrak{irr}=\min\{\pi w((\omega,\tau)):(\omega,\tau)\text{ is a strongly irresolvable $T_3$ topological space}\}
\end{equation*}
\end{crl}

\begin{proof}
Clearly the inequality $\leq$ is true. The other inequality follows from the previous lemma and the fact that the $\pi$-weight of $(X,\tau\upharpoonright X)$ is at most the $\pi$-weight of $(\omega,\tau)$.
\end{proof}

\begin{dfn}
Let $(\omega,\tau)$ be a topological space. The ideal of $\tau$-nowhere dense subsets is denoted by $\mathsf{nwd}(\tau)$.
\end{dfn}

\begin{lemma}\label{saturation_of_nwd_tau}
Let $(\omega,\tau)$ be a strongly irresolvable space. Then $\mathsf{nwd}(\tau)$ is saturated.
\end{lemma}

\begin{proof}
First note that for any $A\in \mathsf{nwd}(\tau)^+$, $A$ has non-empty interior: assume otherwise and pick some $A\in\mathsf{nwd}(\tau)^+$ with empty interior, and let $U\in\tau$ be such that $U\subseteq \bar{A}$. Note that $U\setminus A$ and $U\cap A$ are both dense in $U$, which means that $U$ is resolvable, which is a contradiction. Now, let $\mathcal{A}\subseteq\mathcal{P}(\omega)/\mathsf{nwd}(\tau)$ be a maximal antichain. Then $\{int(A):A\in\mathcal{A}\}$ is also an antichain, and moreover, for any two $A,B\in\mathcal{A}$, $int(A)\cap int(B)=\emptyset$, which implies that $\{int(A):A\in\mathcal{A}\}$ is countable (since $\omega$ is countable), which in turn implies that $\mathcal{A}$ is countable.
\end{proof}

\begin{lemma}\label{destroy_irresolvability}
Let $(\omega,\tau)$ be a strongly irresolvable space. Then $\mathbb{Q}(\mathsf{nwd}(\tau))$ forces that $\dot{x}_{gen}$ and $\omega\setminus\dot{x}_{gen}$ are dense in the topology generated by $\tau$.
\end{lemma}
\begin{proof} Let $p=(H_p,E_p)\in\mathbb{Q}(\mathsf{nwd}(\tau))$ be a condition. Note that for each $n\in dom(E_p)$, we can find extensions $q_0,q_1\leq p$ such that for each $i\in 2$, $H_{q_i}(n)$ has value constant $i$, and $\min(dom(q_i))>n$. Thus, if $V\in\tau$ is a basic open set, we can pick distinct $n_0,n_1\in V\cap dom(E_p)$ and find a condition $q\leq p$ such that for each $i\in 2$, $H_q(n_i)$ has constant value $i$, and $\min(dom(E_q))>\max\{n_0,n_1\}$. This means that $q\Vdash n_0\notin \dot{x}_{gen}\cap V,n_1\in\dot{x}_{gen}\cap V$.
\end{proof}

\noindent\textbf{Remark.} Note that in the generic extension by $\mathbb{Q}(\mathsf{nwd}(\tau))$, if $\tilde{\tau}$ is the topology generated by $\tau$, then the $\pi$-weight of $(\omega,\tilde{\tau})$ is the same as the $\pi$-weight of $(\omega,\tau)$ in the ground model. 
\\

We need the following preservation theorem from \cite{shelah_i_less_u}:
\begin{thm}\label{preservation_theorem}
Let $\mathcal{F}$ be a filter and $\mathcal{D}$ a family of subsets of $\omega$, and let $\mathbb{P}=\langle\mathbb{P}_\beta,\dot{\mathbb{Q}}_\beta:\beta<\alpha\rangle$ a countable support iteration of proper $\omega^\omega$-bounding forcings. Assume the following holds:
\begin{enumerate}
    \item $\mathcal{F}$ is a selective filter.
    \item For any $X\subseteq\omega$ such that $X\notin \mathcal{F}$, there is $D\in\mathcal{D}$ such that $X\subseteq D$.
    \item For all $\beta<\alpha$, $\mathbb{P}_\beta$ forces the statement,
    \begin{equation*}
        (*)\ \ \ \ \ \ \ \ (\forall X\subseteq\omega)( X\notin\langle\mathcal{F}\rangle\Longrightarrow(\exists D\in\mathcal{D})(X\subseteq^* D))
    \end{equation*}
\end{enumerate}
Then $\mathbb{P}$ forces $(*)$ as well.
\end{thm}

Note that if if $\mathcal{F}$ is a selective rational filter, then clauses (1) and (2) of the previous theorem are satisfied with $\mathcal{D}=\{\omega\setminus (B\cap X):B\in\mathcal{B}\land X\in\mathcal{F}\}$. Thus, the previous theorem implies that if we have a countable support iteration of proper forcings such that all its initial segments preserve a selective rational filter $\mathcal{F}$ in the sense of $(*)$ above, then the full iteration also preserves $\mathcal{F}$ as a selective rational filter in the sense of $(*)$. Note that this implies that $\mathcal{F}$ is preserved as a maximal selective filter. This will be used in the next proof.\\

\noindent\emph{Proof of Theorem \ref{main_theorem}.} Let $V$ be a model of $\mathsf{ZFC+GCH}$, and let $\mathscr{F}=\{\mathcal{U}_\alpha:\alpha\in\omega_2\}$ be the family of filters given by Lemma \ref{filters_family}. The forcing is a countable support iteration via a bookeeping of forcings of the form $\mathbb{Q}(\mathsf{nwd}(\tau))$, where $(\omega,\tau)$ is a strongly irresolvable space, taking care that at each successor step of the iteration we destroy the next strongly irresolvable space available in the bookeeping. The property of the family $\mathscr{F}$, makes sure that at any $\alpha\in\omega_2$, in $V[G_\alpha]$, for any saturated ideal $\mathscr{I}$, there are at most countably many filters $\mathcal{U}\in\mathscr{F}$ such that $f^*(\mathscr{I}^*)\cup\mathcal{U}$ generates a filter, where $f$ is an $\mathscr{I}$ to one function. This implies that at any step of the iteration we have destroyed the maximality of at most $\omega_1$ selective rational filters in the family $\mathscr{F}$, and we preserve the maximality of the other $\omega_2$ remaining selective rational filters in the family $\mathscr{F}$. Theorem \ref{preservation_theorem} makes sure that those rational filters in $\mathscr{F}$ which are not destroyed in successor steps of the iteration, are preserved in steps of length a limit ordinal. This in turn implies that the family $\bigcup\mathscr{F}$ remains as a $Dense(\mathbb{Q})$-reaping family in all stages of the iteration. Since every dense subset of $\mathbb{Q}$ appears in an intermediate extension, every dense subset of $\mathbb{Q}$ is $Dense(\mathbb{Q})$-reaped by some set in $\bigcup\mathscr{F}$, so $\mathfrak{r}_\mathsf{nwd}=\omega_1$ holds in the generic extension. Note that we actually have $\mathfrak{r}_\mathbb{Q}^+=\omega_1$. On the other hand, we have that Lemma \ref{destroy_irresolvability} and the bookeeping make sure that we destroy the irresolvability of any strongly irresolvable space that appears in the intermediate extension, so $\mathfrak{irr}>\omega_1$ is true in the generic extension.

\section{$\mathfrak{r}_\mathbb{Q}$ and $\mathfrak{u}_\mathbb{Q}$.}

Let $\mathcal{F}\subseteq Dense(\mathbb{Q})$ be a rational filter. A base for $\mathcal{F}$ is a family  $\mathcal{V}\subseteq\mathcal{F}$ such that any element in $\mathcal{F}$ contains one element from $\mathcal{V}$. The cardinal invariant $\mathfrak{u}_\mathbb{Q}$ is defined as the minimal cardinality of base of a rational filter,
\begin{equation*}
    \mathfrak{u}_\mathbb{Q}=\min\{\vert\mathcal{V}\vert:(\exists\mathcal{F}\subseteq Dense(\mathbb{Q})(\mathcal{F}\text{ is a rational filter}\land\mathcal{V}\text{ is a base for $\mathcal{F}$})\}
\end{equation*}
Clearly $\mathfrak{r}_\mathbb{Q}\leq\mathfrak{u}_\mathbb{Q}$ follows from $\mathsf{ZFC}$. The results from the previous sections can be easily modified to prove the consistency of $\mathfrak{r}_\mathbb{Q}<\mathfrak{u}_\mathbb{Q}$.

\begin{dfn}
The cardinal invariant $\mathfrak{irr}_2$, is defined as the minimal possible $\pi$-weight of a $T_2$ countable irresolvable space whose topology extends $\tau_\mathbb{Q}$, that is,
\begin{equation*}
    \mathfrak{irr}_2=\min\{\pi w((\omega,\tau)):(\omega,\tau)\text{ is a $T_2$ irresolvable space}\}
\end{equation*}
\end{dfn}

It clear that Lemmas \ref{strongly_irr_subspace}, \ref{saturation_of_nwd_tau} and \ref{destroy_irresolvability} apply to $T_2$ irresolvable spaces, so we can destroy irresolvable $T_2$ spaces in the same fashion as we did for regular irresolvable spaces. This remarks are motivated by the following proposition and the corollary below.

\begin{prp}\label{rational_filter_to_irresolvable_space}
$\mathfrak{irr}_2\leq\min\{\mathfrak{u}_\mathbb{Q},\mathfrak{irr}\}.$
\end{prp}

\begin{proof}
Let $\mathcal{B}$ be our fixed base for $\mathbb{Q}$. Let $\mathcal{U}\subseteq Dense(\mathbb{Q})$ a rational filter. Define a base for a topology as $\mathcal{B}(\mathcal{U})=\{B\cap A:B\in\mathcal{B}\land A\in\mathcal{U}\}$ and let $\tau(\mathcal{U})$ be the topology generated by $\mathcal{B}(\mathcal{U})$. Then, $(\mathbb{Q},\tau(\mathcal{U}))$ is a $T_2$ irresolvable space: let $X\subseteq \mathbb{Q}$ be an arbitrary $\tau(\mathcal{U})$-dense set, which in particular implies that $X\in Dense(\mathbb{Q})$, and the definition of $\mathcal{B}(\mathcal{U})$ implies that the intersection of $X$ with any element of $\mathcal{U}$ is a dense set in $\mathbb{Q}$, so the maximality of $\mathcal{U}$ implies that $X\in\mathcal{U}$ holds; this clearly implies that the intersection of any two dense sets is not empty. The space $(\mathbb{Q},\tau(\mathcal{U}))$ is $T_2$ because the base $\mathcal{B}$ separates points. It is also clear that the $\pi$-weight of $(\mathbb{Q},\tau(\mathcal{U}))$ is at most $\chi(\mathcal{U})$. Thus, $\mathfrak{irr}_2\leq\mathfrak{u}_\mathbb{Q}$ follows.

The inequality $\mathfrak{irr}_2\leq\mathfrak{irr}$ is immediate, since any $T_3$ irresolvable space is $T_2$.
\end{proof}

Note that in the previous proposition we can not argue that the topological space $(\mathbb{Q},\tau(\mathcal{U}))$ is a $T_3$ space, since for any clopen subset $B$ of the rationals and any $A\in\mathcal{U}$ we have that $\overline{B\cap A}^{\tau(\mathcal{U})}=B$.

\begin{crl}
It is relatively consistent with $\mathsf{ZFC}$ that $\omega_1=\mathfrak{r}_\mathbb{Q}<\mathfrak{u}_\mathbb{Q}=\omega_2$.
\end{crl}

\section{On the existence of rational $p$ and $q$- filters.}

\begin{thm}
$\mathfrak{d}=\mathfrak{c}$ implies the generic existence of rational $p$-filters.
\end{thm}

\begin{proof}
The argument is essentially the same as for the generic existence of $p$-points. For a given filter $\mathcal{F}\subseteq Dense(\mathbb{Q})$ and $X\in Dense(\mathbb{Q})$, let us say that $X$ is $(\mathcal{F},\mathsf{nwd})$-positive if for all $A\in\mathcal{F}$, $A\cap X\in Dense(\mathbb{Q})$. Let $\mathcal{F}\subseteq Dense(\mathbb{Q})$ be a filter of character less than $\mathfrak{c}$ and let $\mathcal{H}\subseteq\mathcal{F}$ a base for $\mathcal{F}$ of minimal cardinality, fix $\{\vec{D}_\alpha:\alpha\in\mathfrak{c}\}$ an enumeration of all $\subseteq$-decreasing sequences of dense subsets of the rationals. Construct a $\subseteq$-increasing sequence of families $\{\mathcal{F}_\alpha:\alpha\in\mathfrak{c}\}$ such that:
\begin{enumerate}
    \item $\mathcal{F}_0=\mathcal{H}$.
    \item For each $\alpha\in\mathfrak{c}$, if each set in the sequence $\vec{D}_\alpha$ is $(\mathcal{F}_\alpha,\mathsf{nwd})$-positive, then there is $X_\alpha\in\mathcal{F}_{\alpha+1}$ which is almost contained in each term of the sequence $\vec{D}_\alpha$.
    \item For all $\alpha\in\mathfrak{c}$, $\vert\mathcal{F}_\alpha\vert<\mathfrak{c}$.
\end{enumerate}
Assume $\mathcal{F}_\alpha$ has been constructed and each term of $\vec{D}_\alpha=\langle D_n^\alpha:\alpha\in\omega\rangle$ is $(\mathcal{F}_\alpha,\mathsf{nwd})$-positive (in other case define $\mathcal{F}_{\alpha+1}=\mathcal{F}_\alpha$). Let $\langle B_n:n\in\omega\rangle$ be an enumeration of a base for the rationals. For  each $X\in\mathcal{F}_\alpha$, define the function $$\varphi_X(n)=\min\{k\in\omega:(\forall i\leq n)(D^\alpha_n\cap B_i\cap X\cap k\neq\emptyset)\}$$ The family $\{\varphi_X:X\in\mathcal{F}_\alpha\}$ has cardinality less than $\mathfrak{d}$, so there is a function $f_\alpha$ which is not $\leq^*$-dominated by each one of the functions $\varphi_X$. Let $Z_\alpha=\bigcup_{n\in\omega} D_n^\alpha\cap f(n)$. It is easy to see that $Z_\alpha$ is a pseudointersection of $\vec{D}_\alpha$ and $(\mathcal{F}_\alpha,\mathsf{nwd})$-positive, so we can define $\mathcal{F}_{\alpha+1}=\mathcal{F}_\alpha\cup\{Z_\alpha\}$. At limit steps just take the union of all the previously constructed families. 
Now let $\mathcal{G}$ be the filter generated by $\bigcup_{\alpha\in\mathfrak{c}}\mathcal{F}_\alpha$. Clause (2) makes sure $\mathcal{G}$ is a $p$-filter. To see that $\mathcal{G}$ is maximal, pick $D\in Dense(\mathbb{Q})$, and let $\alpha$ be such that $\vec{D}_\alpha$ is the constant sequence with values equal to $D$, then either, $D$ is not $(\mathcal{F}_\alpha,\mathsf{nwd})$-positive, or $Z_\alpha\subseteq^* D$.
\end{proof}

Let us recall that a function $f:2^{<\omega_1}\to X$, where $X$ is a Polish space, is Borel, if for all $\alpha\in\omega_1$, the restriction $f\upharpoonright 2^\alpha$ is a Borel function. The use of parametrized diamond principles has shown to be useful when establishing the existence of some specific mathematical objects, such as $p$-points, Souslin trees, independent families, ideal independent families, etc. In \cite{balzar_hrusak_hernandez}, it was proved that $\diamondsuit(\mathfrak{r}_{\mathbb{Q}})$ implies the existence of an independent family. Here we prove that the diamond principle $\diamondsuit(\mathfrak{r}_{\mathbb{Q}}^+)$ implies the existence of selective rational filters. Let us briefly recall the diamond principle $\diamondsuit(\mathfrak{r}_\mathbb{Q}^+)$:
\begin{enumerate}
    \item [$\diamondsuit(\mathfrak{r}_\mathbb{Q}^+)$:] For any Borel function $F:2^{<\omega_1}\to Dense(\mathbb{Q})$, there is a function $g:\omega_1\to Dense(\mathbb{Q})$ such that for any $f\in 2^{\omega_1}$, the set $$\{\alpha\in\omega_1: g(\alpha)\text{ $Dense(\mathbb{Q})^+$-reaps }F(f\upharpoonright\alpha)\}$$ is stationary.
\end{enumerate}

\begin{thm}
$\diamondsuit(\mathfrak{r}_{\mathbb{Q}}^+)$ implies the existence of a selective rational filter $\omega_1$-generated. 
\end{thm}

\begin{proof}
It is not hard to see that the functions described below can be chosen to be Borel, so we omit the details of this. First let us recall that if $\diamondsuit(\mathfrak{r}_\mathbb{Q}^+)$ holds, then $\mathfrak{r}_\mathbb{Q}^+=\omega_1$, and since $\mathfrak{d}\leq\mathfrak{r}_\mathbb{Q}^+$ holds, we get $\mathfrak{d}=\omega_1$ as well. Let $\{\vec{\mathcal{I}}_\alpha:\alpha\in\omega_1\}$ be a dominating family of interval partitions, that is for any interval partition $\vec{\mathcal{J}}$, there is $\alpha\in\omega_1$ such that any $I\in\vec{\mathcal{I}}_\alpha$ contains one element from $\vec{\mathcal{J}}$. For each limit ordinal $\alpha\in\omega_1$, fix  a bijection $e_\alpha:\omega\to\alpha$. Now, by a suitable coding, we can assume that the domain of the function $F$ is $dom(F)=\bigcup_{\alpha\in\lim(\omega_1)}[\omega]^{\omega}\times([\omega]^\omega)^\alpha$. Given a sequence $\vec{X}=\langle X_\beta:\beta\in\alpha\rangle$, let us say that is $\mathbb{Q}$-centered if the intersection of finitely many terms of the sequence is a dense set. Given a $\mathbb{Q}$-centered sequence $\vec{X}=\langle X_\beta:\beta<\alpha\rangle$, let $\varphi_\alpha(\langle X_\beta:\beta<\alpha\rangle)$ be the set $\{k_n:n\in\omega\}$ defined by,
\begin{equation*}
    k_n=\min\left(\bigcap_{j\leq n}X_{e_\alpha(j)}\cap B_n\right)
\end{equation*}
Note that the set $\{k_n:n\in\omega\}$ is a dense set and a pseudointersection of $\{X_\beta:\beta<\alpha\}$. Now define the function $F$ as follows:
\begin{enumerate}
    \item[a)] If $(C,\langle X_\beta:\beta<\alpha\rangle)$ is such that $\langle X_\beta:\beta<\alpha\rangle$ $\mathbb{Q}$-centered sequence, let $F(C,\langle X_\beta:\beta<\alpha\rangle)=\omega$.
    \item[b)] If $\langle X_\beta:\beta<\alpha\rangle$ $\mathbb{Q}$-centered sequence, but $C\cap\varphi_\alpha(\langle X_\beta:\beta<\alpha\rangle)$ is not a dense set, define $F(C,\langle X_\beta:\beta<\alpha\rangle)=\omega$.
    \item[c)] In the remaining case, let $\psi_{\varphi_\alpha(\langle X_\beta:\beta\in\alpha\rangle)}:\omega\to\varphi_\alpha(\langle X_\beta:\beta<\alpha\rangle)$ be a homeomorphism between $\omega$ and $\varphi_\alpha(\langle X_\beta:\beta<\alpha\rangle)$, which is constructed in a recursive way. Then define $F(C,\langle X_\beta:\beta<\alpha\rangle)=\psi_\alpha^{-1}[C\cap\varphi_\alpha(\langle X_\beta:\beta<\alpha\rangle)]$.
\end{enumerate}

Let now $g$ be a guessing function for $F$. Define an increasing sequence of families $\{\mathcal{F}_\alpha:\alpha\in \lim(\omega_1)\}$ as follows:
\begin{enumerate}
    \item Start with $\mathcal{F}_\omega=\{\omega\setminus k:k\in\omega\}$.
    \item Suppose $\mathcal{F}_\alpha=\{D_\beta:\beta<\alpha\}$ has been defined. Let $D_\alpha\subseteq\psi_{\varphi_\alpha(\langle D_\beta:\beta\in\alpha\rangle)}[g(\alpha)]$ be such that it is a dense set and a partial selector of $\vec{\mathcal{I}}_\alpha$, and put $\mathcal{F}_{\alpha+1}=\mathcal{F}_\alpha\cup\{D_\alpha\}$.
\end{enumerate}
It is clear from the construction that $\mathcal{F}_\alpha$ is a family of dense sets generating a filter. Also note that $\mathcal{H}=\{D_\alpha:\alpha\in\omega_1\}$ is a $\subseteq^*$-decreasing sequence of dense subsets, so it is enough to prove that the filter generated by $\mathcal{H}$ is maximal and a $q$-filter. Let $X\in Dense(\mathbb{Q})$ be arbitrary and consider $(X,\langle D_\alpha:\alpha\in\omega_1\rangle)$. By the guessing property of $g$, there is $\alpha\in\omega_1$ such  that $g(\alpha)$ $Dense(\mathbb{Q})$-reaps $F(X,\langle D_\beta:\beta<\alpha\rangle)$, that is $g(\alpha)$ $Dense(\mathbb{Q})^+$-reaps the set $\psi_{\varphi_\alpha(\langle D_\beta:\beta<\alpha\rangle)}^{-1}[X\cap\varphi_\alpha(\langle D_{\beta}:\beta<\alpha\rangle)]$, so we have that, either $D_\alpha\subseteq^*X$ or $D_\alpha\cap X$ is not dense. This shows that the filter generated by $\mathcal{H}$ maximal. Since it is generated by a $\subseteq^*$-decreasing sequence, it is also a $p$-filter. To see that this filter is a $q$-filter, note that by clause (2) above, $D_\alpha$ is a partial selector of $\vec{\mathcal{I}}_\alpha$; this clearly implies that the filter generated by $\mathcal{H}$ has partial selector for all interval partitions.
\end{proof}

Now we are aiming to prove that consistently there is no rational $p$-filter. It may be expected that the proof should be quite similar to the one for $p$-points. However, after trying to adapt the proof one finds some difficulties which arise from the fact that we don't have the property ``$x\in\mathcal{F}$ or $\omega\setminus x\in\mathcal{F}$''. Indeed, the typical forcings used to destroy $p$-points, o specific non-meager $p$-filters, add an $\omega$-sequence $\langle\dot{x}_n:n\in\omega\rangle$ of $\mathcal{F}$-positive sets, and any ultrafilter extending $\mathcal{F}$ makes a selection of these sets and its complements, and the forcings are designed to make sure that all pseudointersections of such selection live in the dual ideal $\mathcal{F}^*$. In the present context we can have actually that if $\mathcal{G}$ is a rational filter extending $\mathcal{F}$, then $\{\dot{x}_n,\omega\setminus\dot{x}_n:n\in\omega\}\cap\mathcal{G}=\emptyset$, that is, neither $\dot{x}_n$ and $\omega\setminus \dot{x}_n$ are in $\mathcal{G}$, for all $n\in\omega$. Indeed, assume $\{B_n:n\in\omega\}$ is a maximal antichain of basic open sets, and let us take, for example, the set
$$X=\bigcup_{n\in\omega}\dot{x}_n\cap\dot{x}_{n-1}^{0}\cap\ldots\cap\dot{x}_0^0\cap B_n$$
where $\dot{x}_n^0=\omega\setminus \dot{x}_n$, and $\dot{x}_n$ are the generic sets added by any of the typical forcings used to destroy non-meager $p$-filters. If $\mathcal{G}$ is a filter extending $\mathcal{F}\cup\{X\}$, then $\mathcal{G}$ is one of the stated filters. After thinking a little bit more, one finds that infinite sequences of sets similar to $X$ above can make things more complicated. Thus, this path, although not completely proved to be unprofitable, seems hard to take over. 

The answer to the problem comes from a theorem proved in \cite{no_p_saturated_filters}: it is consistent that there is no saturated $p$-filter. Recall from Proposition \ref{rational_filter_to_irresolvable_space} that from a rational filter $\mathcal{F}$ we constructed an irresolvable space $(\omega,\tau(\mathcal{F}))$. Now, Lemma \ref{strongly_irr_subspace} implies that there is a $\tau(\mathcal{F})$-open set $V$ such that $(V,\tau(\mathcal{F})\upharpoonright V)$ is strongly irresolvable. We claim that $\mathcal{F}\upharpoonright V=\{X\cap V:X\in\mathcal{F}\}$ is a saturated $p$-filter.

\begin{prp}
If $\mathcal{F}$ is a rational $p$-filter, then there is a basic open set $B\in \mathcal{B}$ and some $X\in\mathcal{F}$ such that $F\upharpoonright(X\cap B)$ is a saturated $p$-filter.    
\end{prp}

\begin{proof}
We have seen in the paragraph preceding this proposition the existence of a $\tau(\mathcal{F})$-open set $V$ such that $\mathcal{F}\upharpoonright V$ is strongly irresolvable. A base the topology $\tau(\mathcal{F})$ was defined as $\mathcal{B}_{\mathcal{F}}=\{B\cap X:B\in\mathcal{B}\land X\in\mathcal{F}\}$, so we can find $B\in\mathcal{B}$ and $X\in\mathcal{F}$ such that $X\cap B\subseteq V$. It is clear that $\mathcal{F}\upharpoonright(X\cap B)$ is a $p$-filter. To prove that it is also a saturated filter, by Lemma \ref{saturation_of_nwd_tau}, it is enough to prove that $\mathsf{nwd}(\tau(\mathcal{F}))\upharpoonright(X\cap B)^*\subseteq \mathcal{F}\upharpoonright(X\cap B)$. Note that a base for the topology $\tau(\mathcal{F})\upharpoonright(X\cap B)$ is given by $\mathcal{B}_\mathcal{F}'=\{W\in\mathcal{B}_\mathcal{F}:W\subseteq X\cap B\}$. For simplicity, let us write $\tilde{\tau}$ instead of $\tau(\mathcal{F})\upharpoonright(X\cap B)$. Note that since $X\cap B$ is $\tau(\mathcal{F})$-open, we have that $N\subseteq X\cap B$ is $\tilde{\tau}$-nowhere dense if and only if it is $\tau(\mathcal{F})$-nowhere dense, so $\mathsf{nwd}(\tilde{\tau})=\mathsf{nwd}(\tau(\mathcal{F}))\upharpoonright(X\cap B)$.

Pick $N\in\mathsf{nwd}(\tilde{\tau})$, we want to show that $(X\cap B)\setminus N\in\mathcal{F}\upharpoonright(X\cap B)$.
Since $N$ is nowhere dense, any open set $W\in\tau(\mathcal{F})\upharpoonright(X\cap V)$ can be extended to a basic open set having empty intersection with $N$. Thus, we can find a maximal antichain of basic open sets $\{B_n\cap X_n:n\in\omega\}\subseteq\mathcal{B}_\mathcal{F}'$, each one of whose elements has empty intersection with $N$. Note that by the definition of the basis $\mathcal{B}_\mathcal{F}$ and $\mathcal{B}_\mathcal{F}'$, we should have $B_n\cap B_m=\emptyset$ whenever $n\neq m$. Note also that $\bigcup_{n\in\omega}B_n$ is an open dense subset of $B$ (in the usual topology of the rationals). Since $\mathcal{F}$ is a $p$-filter, there is $X_\omega$ which is almost contained in each one of the sets $X_n$. Then $F_n=X_\omega\cap B_n\setminus X_n\cap B_n$ is a finite subset of $B_n$; this implies that $\bigcup_{n\in\omega}F_n$ is a discrete set in $\bigcup_{n\in\omega}B_n$, so $Z=X_\omega\setminus \bigcup_{n\in\omega}F_n\in\mathcal{F}$. Then we have that $Z\cap B_n$ is a basic open set and is contained in $X_n\cap B_n$, so it has empty intersection with $N$. On the other hand, $W=\bigcup_{n\in\omega}Z\cap B_n$ is an open $\tilde{\tau}$-dense set, and it is an element of the filter $\mathcal{F}\upharpoonright(X\cap B)$: note that $U=(\omega\setminus \overline{B})\cup\bigcup_{n\in\omega}B_n$ is an open dense set of $\mathbb{Q}$, so we have $U\in\mathcal{F}$, which implies $Z\cap U\in\mathcal{F}$, and it is easy to see that $B\cap Z\cap U=W$. Thus, we conclude that $W\in\mathcal{F}\upharpoonright(X\cap B)$, which in turn implies $\omega\setminus N\in\mathcal{F}\upharpoonright(X\cap B)$ (because $W$ and $N$ are disjoint). We have proved that $\mathsf{nwd}(\tilde{\tau})^*\subseteq\mathcal{F}\upharpoonright(X\cap B)$. 
\end{proof}

\begin{crl}
The existence of a rational $p$-filter implies the existence of a saturated $p$-filter.    
\end{crl}
\begin{crl}
It is relatively consistent that there are no rational $p$-filters.    
\end{crl}
\begin{proof}
In the model constructed in \cite{no_p_saturated_filters} there are no saturated $p$-filters, so the previous corollary implies that in such model there is no rational $p$-filter as well.    
\end{proof}

Finally, note that since in the Laver model there is no $q$-point, there are also neither rational $q$-filters: if it were the case that there is a rational $q$-filter, just extend it to an ultrafilter on $\mathbb{Q}$, which results in a $q$-point.

\begin{thm}
In the Laver model there are no rational $q$-filters.    
\end{thm}

\bibliographystyle{alpha}
\bibliography{bibliography}
\end{document}